\documentclass{amsart}
\usepackage{epsfig}
\usepackage{amsmath}
\usepackage{amssymb}
\usepackage{amscd}
\usepackage{graphicx}
\usepackage{color}
\usepackage{float}
\usepackage{verbatim}
\usepackage{bbm}
\usepackage{enumerate}
\usepackage{hyperref}
\numberwithin{equation}{section}

\newtheorem{thm}{Theorem}[section]
\newtheorem{lem}[thm]{Lemma}

\newtheorem{fact}[thm]{Fact}

\newtheorem{cor}[thm]{Corollary}
\newtheorem{defn}[thm]{Definition}

\theoremstyle{remark}
\newtheorem{rem}[thm]{Remark}

\def \N {\mathbb N}
\def \T {\mathcal T}
\def \TT {\mathsf T}
\def \TTT {\boldsymbol{\mathsf T}}
\def \bT {\boldsymbol\T}
\def \O {\mathcal O}
\def \Z {\mathbb Z}

\def \Q {\mathcal Q}

\def \P {\mathcal P}
\def \S {\mathcal S}

\def \B {\mathcal B}
\def \I {\mathcal I}

\def \eps {\varepsilon}

\def \sq {sequence}

\def \tl {topological}
\def \im {invariant measure}
\def \inv {invariant}

\def \htop {h_{\mathsf{top}}}

\begin{document}
	\title{Asymptotic pairs in topological actions of amenable groups}
	
	\author[T.\ Downarowicz and M.\ Wi\c{e}cek]{Tomasz Downarowicz and Mateusz Wi\c{e}cek}
	
	\address{\vskip 2pt \hskip -12pt Tomasz Downarowicz}
	
	\address{\hskip -12pt Faculty of Pure and Applied Mathematics, Wroc\l aw University of Science and Technology, Wroc\l aw, Poland}
	
	\email{downar@pwr.edu.pl}
	
	\address{\vskip 2pt \hskip -12pt Mateusz Wi\c{e}cek}
	
	\address{\hskip -12pt Faculty of Pure and Applied Mathematics, Wroc\l aw University of Technology, Wroc\l aw, Poland}
	
	\email{mateusz.wiecek@pwr.edu.pl}
	
	\begin{abstract}
		We provide a definition of a $\prec$-asymptotic (we suggest the pronunciation ``prec-asymptotic'') pair in a topological action of a countable group $G$, where $\prec$ is an order on $G$ of type $\Z$. We then prove that if $G$ is a countable amenable group and $(X,G)$ is a \tl\ $G$-action of positive entropy, then for every multiorder
		$(\tilde{\O},\nu,G)$ and $\nu$-almost every order $\prec\,\in\tilde{\O}$ there exists a $\prec$-asympotic pair in~$X$. This result is a generalization of the Blanchard-Host-Ruette Theorem for classical topological dynamical systems (actions of~$\Z$). We also prove that for every countable amenable group $G$, and every multiorder on $G$ arising from a tiling system, every topological $G$-action of entropy zero has an extension which has no $\prec$-asymptotic pairs for any $\prec$ belonging to this multiorder. Together, these two theorems give a characterization of topological $G$-actions of entropy zero: $(X,G)$ has topological entropy zero if and only if, for any multiorder $\tilde\O_{\TTT}$ on $G$ arising from a tiling system of entropy zero, there exists an extension $(Y,G)$ of $(X,G)$, which has no $\prec$-asymptotic pairs for any $\prec\,\in\tilde\O_{\TTT}$, equivalently, there exists a multiorder $(\tilde\O,\nu,G)$ on $G$, such that for $\nu$-almost any $\prec\,\in\tilde\O$, there are no $\prec$-asymptotic pairs in $(Y,G)$.
	\end{abstract}
	
	\maketitle
	\section{Introduction}
	In a classical topological dynamical system $(X,T)$, where $X$ is a compact metric space with a metric $d_X$ and $T:X\to X$ is a homeomorphism, an \emph{asymptotic pair} is defnined as a pair of distinct points $x,x'\in X$ such that
	\begin{equation}
		\lim\limits_{k\to+\infty}d_X(T^k(x),T^k(x'))=0.
	\end{equation}
	The set of all aymptotic pairs in $(X,T)$ does not depend on the choice of the metric~$d_X$.
	
	Using the notion of asymptotic pairs it is possible to characterize topological $\Z$-actions of entropy zero. Firstly, in 2002, F.~Blanchard, B.~Host and S.~Ruette proved the following theorem.
	\begin{thm}[Blanchard-Host-Ruette]\cite[Proposition~1]{BHR}
		Let $(X,T)$ be an invertible topological dynamical system with positive topological entropy. Then $(X,T)$ has asymptotic pairs.
		More precisely, the set of points belonging to asymptotic pairs has measure $1$ for any ergodic measure on $X$ with positive entropy.
	\end{thm}
	Later, in 2010, T.~Downarowicz and Y.~Lacroix proved that every topological dynamical system of entropy zero is a factor of a system with no asymptotic pairs~\cite[Lemma~4.3]{DL}. These two results combined give a full characterization of zero-entropy systems as factors of systems with no asymptotic pairs. There have been many attempts to generalize this characterization for actions of groups other than $\Z$. Any such attempt requires providing a definition of an asymptotic pair fitting the more general setup.
	So far, the best results were obtained by W. Huang, L. Xu and Y. Yi in 2014 \cite{HuXuYi}, and by W. Bułatek, B. Kamiński and J. Szymański in 2016 \cite{BuKa}. These two teams of authors independtently proved that if $G$ is an infinite amenable group which is orderable, then for every topological $G$-action of positive entropy, there exists a pair which is asymptotic with respect to the invariant order (note that this order need not be of type~$\Z$). However, none of the definitions provided so far can be used in the case of an arbitrary countable amenable group (which need not admit an invariant order). Recently T. Downarowicz, P. Oprocha, M. Wi\c{e}cek and G. Zhang considered orders of type~$\Z$ on countable groups. This approach opens a  possibility of creating a new definition of an asymptotic pair in systems with actions of countable groups, generalizing the classical one. A multiorder, introduced in~\cite{DOWZ}, is a collection of orders of type~$\Z$, invariant under a specific action of the group and supporting an invariant probability measure~$\nu$.
	A special type of multiorders, also described in~\cite{DOWZ} and in this paper called \emph{tiling-based multiorders} can be viewed as a \tl\ counterpart of multiorders (which by definition depend on an \im).  With the aid of multiorders of these two types, it is possible to generalize both the Blanchard-Host-Ruette and Downarowicz-Lacroix theorems and provide a full characterization of \tl\ $G$-actions of entropy zero in terms of asymptotic pairs, where $G$ is an arbitrary countable amenable group.
	
	Section~\ref{s2} of this paper contains the definitions of a \emph{multiorder} and a \emph{multiordered dynamical system}, as well as some useful theorems concerning the properties of these notions.
	
	Section~\ref{s3} begins with the definitions of a \emph{$\prec$-asymptotic pair} in a topological dynamical system with an action of a countable group $G$, where $\prec$ is an order of type $\Z$ on $G$, and a \emph{$\varphi$-asymptotic pair} in a topological system which has a multiorder as a measure-theoretic factor for some $G$-invariant measure. These definitions are followed by the formulation and proof of one of the main theorems of this paper, the generalization of the Blanchard-Host-Ruette Theorem.
	
	Section~\ref{s4} consists of preliminaries necessary to formulate and prove the series of theorems presented in Section~\ref{section_extension}. It starts with a short review of principal extensions, zero-dimensional systems and array systems. The rest of the section is devoted to tilings, tiling systems and tiling-based multiorders together with some important properties of these objects. We also define a \emph{topologically multiordered dynamical system} as a topological system which factors to a tiling-based multiorder through a tiling system. We end this section by introducing a new class of \emph{odometric tiling systems}.
	
	In Section~\ref{section_extension} we present a series of theorems leading to a characterization of zero-entropy systems in terms of asymptotic pairs. The key theorem of this section asserts that for every \tl\ $G$-action $(X,G)$ of entropy zero (where $G$ is a countable amenable group) and every tiling-based multiorder $\tilde{\O}_{\TTT}$ of entropy zero there exists an extension $(Y,G)$ of $(X,G)$ which is topologically multiordered by $\tilde{\O}_{\TTT}$ and has no $\prec$-asymptotic pairs for any  $\prec\,\in\tilde{\O}_{\TTT}$. This, combined with the main result of Section~\ref{s3}, allows us to conclude the paper by stating a full characterization of zero-entropy topological actions of countable amenable groups
	as ``factors of systems with no asymptotic pairs'', which generalizes \cite[Theorem~4.1]{DL}.
	
	\section{Multiorders and multiordered dynamical systems}\label{s2}
	This section is devoted to familiarize the reader with the definitions of a multiorder and a multiordered dynamical system and with some important properties of these objects. The section is organized into three subsections. In the first one we provide the general definition of a multiorder on any countable group $G$ and its properties valid in any countable group $G$. In the second subsection, we focus on countable amenable groups, we remind the definition of an amenable group using the notion of a F{\o}lner sequence, and present some crucial properties of multiorders on such groups. Finally, the third subsection is devoted to multiordered dynamical systems and their specific properties. In the entire section we skip the proofs of all theorems and lemmas, which (unless indicated otherwise) can be found in the paper \cite{DOWZ}.
	\subsection{Multiorders}
	\begin{defn}
		Let $G$ be a countable set. An \emph{order of type $\Z$} on $G$ is a total order\footnote{By a total order on $G$ we mean a transitive relation $\prec$ such that for every $a,b\in G$ exactly one of the alternatives holds: either $a\prec b$ or $b\prec a$, or $a=b$.} $\prec$ on $G$ such that all order intervals $[a,b]^{\prec}=\{a,b\}\cup\{g\in G: a\prec g\prec b\}$ ($a,b\in G,\ a\prec b$) are finite and there are neither minimal nor maximal elements in~$G$.
	\end{defn}
	\begin{rem}
		In case there exists a maximal (but not minimal) element in $G$, or the other way round, we say that the order is of type $-\N$ or $\N$, respectively
	\end{rem}
	The set $\tilde{\O}$ of all orders on $G$ of type $\Z$ is a subset of the family of all relations on $G$ which can be viewed as the topological space $\{0,1\}^{G\times G}$. Hence, $\tilde{\O}$~inherits from $\{0,1\}^{G\times G}$ a natural topological and Borel structure. When $G$ is a countable group, $G$~acts on $\tilde{\O}$ by homeomorphisms as follows: for $g\in G$ and $\prec\,\in\tilde{\O}$ we have $g(\prec)=\,\prec'$ where $\prec'$ is given by
	\begin{equation}\label{ac}
		a\prec' b \iff ag\prec bg.
	\end{equation}
	\begin{defn}\label{def_multiorder}
		Let $\nu$ be a $G$-invariant Borel probability measure supported by $\tilde\O$. By a \emph{multiorder} (on $G$) we mean the measure-preserving $G$-action $(\tilde\O,\Sigma_{\tilde\O},\nu,G)$, where $\Sigma_{\tilde\O}$ is the Borel sigma-algebra on $\tilde\O$.
	\end{defn}
	For brevity, in what follows, we will write $(\tilde{\O},\nu,G)$ instead of $(\tilde{\O},\Sigma_{\tilde{\O}},\nu,G)$. Note that the natural order $<$ on $\Z$ is a fixed point of the action of $\Z$ given by the additive version of formula~\eqref{ac}. Hence the system $(\{<\}, \delta_{<},\Z)$, where $\delta_{<}$ is the Dirac measure at $<$, is a multiorder. However, there exist many more multiorders on $\Z$ consisting of non-standard orders (see e.g.\ \cite[Example~B.9]{DOWZ}).
	
	There are other ways of representing multiorders on countable groups. Especially useful is representing a multiorder as a family of \emph{anchored bijections} from $\Z$ to $G$.
	\begin{defn}\label{ddd}
		With each $\prec\ \in\tilde\O$ we associate the bijection $\mathsf{bi}_\prec:\Z\to G$ which is \emph{anchored} (i.e., satisfies $\mathsf{bi}_\prec(0)=e$, where $e$ is the unit of $G$), and on the rest of $\Z$ is determined by the property:
		\begin{equation}\label{bsu}
			\mathsf{bi}_\prec(i) = g \iff \mathsf{bi}_\prec(i+1)= \mathsf{succ}_\prec(g), \ 
			i\in\Z,\ g\in G,
		\end{equation}
		where $\mathsf{succ}_\prec(g)$ denotes the successor of $g$ with respect to $\prec$.
		
		On the set $\mathbf{Bi}(\Z,G)$ of anchored bijections $\mathsf{bi}:\Z\to G$ we define an action of $G$ by the following formula:
		\begin{equation}\label{bij}
			(g(\mathsf{bi}))(i)= \mathsf{bi}(i+k)\cdot g^{-1}, \text{ where $k$ is such that } g=\mathsf{bi}(k).
		\end{equation}
	\end{defn}
	By a straightforward calculation, it can be showed that $\mathbf{Bi}(\Z,G)$ is a $G_{\delta}$-subset of the compact space ${\overline{G}}^{\Z}$, where $\overline{G}$ is the Alexandroff Compactification of~$G$. Hence $\mathbf{Bi}(\Z,G)$ is a Polish and, in an appropriate metric, totally bounded space (this fact will be essentially used in Section~3).
	\begin{thm}\cite[Proposition 2.11]{DOWZ}\label{odw1} The assignment $\psi:\tilde \O\to \mathbf{Bi}(\Z,G)$, given by 
		${\psi(\prec)=\mathsf{bi}_\prec}$, is a measurable bijection with a continuous inverse, which intertwines the action of $G$ on $\tilde\O$ given by \eqref{ac} with the action of $G$ on $\mathbf{Bi}(\Z,G)$ given by~\eqref{bij}. 
	\end{thm}
	By convention, if $\nu$ is an invariant measure on $\tilde{\O}$ we will denote the push-forward measure on $\mathbf{Bi}(\Z,G)$ by the same letter $\nu$. We will denote by $k^{\prec}$ the element $\mathsf{bi}_{\prec}(k)$. Using this notation we have:
	\begin{cor}
		For every order $\prec\,\in\tilde{\O}$ and every $i\in\Z$ and $g\in G$ we have
		\begin{equation}\label{16}
			i^{g(\prec)} = (i+k)^{\prec}\cdot g^{-1}, \text{ \ equivalently \ }i^{\prec}\cdot g^{-1}= (i-k)^{g(\prec)},
		\end{equation}
		where $k$ is the unique integer such that 
		\begin{equation}\label{17}
			g=k^\prec, \text{ \ equivalently \ } g^{-1}=(-k)^{g(\prec)}.
		\end{equation}
	\end{cor}
	\subsection{Multiorders on countable amenable groups}	
	From now on, $G$ will always denote a countable, discrete group.
	
	\begin{defn}
		Fix an $\varepsilon>0$ and let $K\subset G$ be finite. A finite set $F\subset G$ is called $(K,\varepsilon)$-invariant, if it satisfies
		\begin{equation*}
			\frac{|KF\triangle F|}{|F|}<\varepsilon.
		\end{equation*}
	\end{defn}
	
	\begin{defn}
		By a F{\o}lner sequence on $G$, we mean a sequence $(F_n)_{n\in\N}$ of finite subsets of $G$, such that for every finite set $K\subset G$ and every $\varepsilon>0$, for all sufficiently large indices $n$, the sets $F_n$ are $(K,\varepsilon)$-invariant.
	\end{defn}
	Using the notion of a F{\o}lner sequence, we may formulate a simple definition of amenability in case the group $G$ is discrete and countable.
	\begin{defn}
		A discrete, countable group $G$ is called \emph{amenable} if there exists a F{\o}lner sequence on $G$.
	\end{defn}
	
	\begin{thm}\label{exi}\cite[Theorem 2.6]{DOWZ}
		Let $G$ be a countable amenable group. There exists a multiorder $(\tilde\O,\nu,G)$ of (measure-theoretic) entropy zero. 
	\end{thm}
	
	The key property of multiorders on countable amenable groups (apart from the guaranteed existence on such groups) is that order-intervals of almost every order $\prec$ from a multiorder form a F{\o}lner sequence. In \cite{DOWZ} this phenomenon is called \emph{F{\o}lner property} of a multiorder and it was proved (see \cite[Theorem~2.5]{DOWZ}) that it is automatically satisfied by every multiorder on a countable amenable group. However, for every countable amenable group $G$, it is possible to construct (using a system of tilings) a specific multiorder on $G$ which exhibits a stronger version of the F{\o}lner property, called the \emph{uniform F{\o}lner property}.
	
	\begin{defn}\label{uniform_folner}
		We say that a multiorder $({\tilde{\O}},\nu,G)$ has \emph{uniform F{\o}lner property} if for every finite set $K\subset G$ and every $\varepsilon>0$ there exists $N\in\mathbb{N}$ such that, for $\nu$-almost every $\prec\,\in\tilde{\O}$, any order interval of length at least $N$ is $(K,\varepsilon)$-invariant.
	\end{defn}
	\begin{thm}\cite[Corollary~B.7]{DOWZ}\label{thm_uniform}
		For every countable amenable group $G$ there exists a uniformly F{\o}lner multiorder on $G$ of entropy zero.
	\end{thm}
	
	\subsection{Multiordered dynamical systems}
	\begin{defn}
		By a \emph{multiordered dynamical system}, denoted by $(X,\mu,G,\varphi)$, we will mean a measure-preserving $G$-action $(X,\mu,G)$ with a fixed measure-theoretic factor map ${\varphi:(X,\mu,G)\to(\tilde\O,\nu,G)}$ to a multiorder on $G$ (provided such a factor map exists).
	\end{defn}
	Let $G$ and $\Gamma$ be two countable groups acting on the same probability space $(X,\mu)$. These actions are called \emph{orbit equivalent} if for $\mu$-almost every $x\in X$ we have $\{g(x)\,:\,g\in G\}=\{\gamma(x)\,:\,\gamma\in\Gamma\}$.
	
	It turns out that every multiordered $G$-action is orbit equivalent to a specific $\Z$-action such that the orbit equivalence preserves the multiorder factor, as the theorem below states.
	
	\begin{thm}\label{motooe}\cite[Theorem 3.5]{DOWZ}
		Suppose $\varphi:X\to\tilde\O$ is a measure-theoretic factor map from a measure-preserving $G$-action $(X,\mu,G)$ to a multiorder $(\tilde\O,\nu,G)$. Then $(X,\mu,G)$ is orbit-equivalent to the $\Z$-action generated by the \emph{successor map} defined as follows:
		\begin{equation}\label{full}
			S(x) = 1^\prec(x), \text{ \ where \ }\prec\ = \varphi(x),
		\end{equation}
		i.e., $S(x)=g(x)$, where $g=1^{\varphi (x)}$.
		Moreover, for any $k\in\Z$, we have
		\begin{equation}\label{TD}
			S^k(x)=k^\prec(x).
		\end{equation}
		Let $\tilde{S}$ denote the transformation on $\tilde{\O}$ defined by
		\begin{equation}\label{SforO}
			\tilde{S}(\prec)=1^{\prec}(\prec),
		\end{equation}
		i.e., $\tilde{S}(\prec)=g(\prec)$, where $g=1^{\prec}$ and $g(\prec)$ is given by the formula \eqref{ac}.
		Then $\tilde{S}$ preserves the measure $\nu$, the $\Z$-action on $\tilde{\O}$ generated by $\tilde S$ is orbit equivalent to the $G$-action given by \eqref{ac}, and $\varphi$ is a factor map from the $\Z$-action $(X,\mu,S)$ to the $\Z$-action $(\tilde{\O},\nu,\tilde{S})$.
	\end{thm}
	Additionally, the above orbit equivalence between $(X,\mu,G)$ and $(X,\mu,S)$ preserves the conditional entropy with respect to the multiorder factor, which is captured by the following theorem.
	\begin{thm}\label{eq_entropies}\cite[Theorem 5.1]{DOWZ}
		Let $(X,\mu,G,\varphi)$ be a multiordered dynamical system and let $S$ denote the successor map defined by the formula \eqref{full}. Then, for every finite, measurable partition $\P$ of $X$ we have
		\begin{equation}
			h(\mu,G,\P|\Sigma_{\tilde{\O}})=h(\mu,S,\P|\Sigma_{\tilde{\O}}),
		\end{equation}
		where $h(\mu,G,\P|\Sigma_{\tilde{\O}})$ is the conditional (with respect to $\Sigma_{\tilde\O}$) entropy of the process $(X,\mu,\P,G)$ generated by $\P$ under the action of $G$ and $h(\mu,S,\P|\Sigma_{\tilde{\O}})$ is the analogous conditional entropy of the process $(X,\mu,\P,S)$ generated by $\P$ under the action of $\Z$ given by the iterates of $S$.
	\end{thm}
	\section{Asymptotic pairs in positive entropy actions of countable amenable groups}\label{s3}
	We begin by providing the definition of a $\prec$-asymptotic pair in a topological action of a general countable group $G$. If not indicated otherwise, by a topological $G$-action $(X,G)$ we mean a compact metric space $(X,d_X)$ on which a countable group $G$ acts by homeomorphisms.
	\begin{defn}
		Let $(X,G)$ be a topological action of a countable group $G$. Let $\prec$ be an order of type $\Z$ on $G$. A pair of distinct points $x,x'\in X$ is called \emph{$\prec$\,-asymptotic} if it satisfies
		\begin{equation}\label{conv}
			\lim\limits_{k\to+\infty}d_X(k^{\prec}(x),k^{\prec}(x'))=0
		\end{equation}
		($k^{\prec}$ denotes the $k$-th element of $G$, counting from $e$, along the order~$\prec$).
	\end{defn}
	\begin{defn}\label{phi-asympt}
		Let $(X,G)$ be a topological action of a countable group $G$, such that for some $G$-invariant Borel probability measure $\mu$ on $X$ there exists a multiorder $(\tilde{\O},\nu,G)$ being a (measure-theoretic) factor of the system $(X,\mu,G)$ via a factor map $\varphi$. If two distinct points $x,x'$ of $X$ satisfy $\varphi(x)=\varphi(x')=\,\prec$ and $x,x'$ form a $\prec$-asymptotic pair, then we say that the pair $(x,x')$ is $\varphi$-asymptotic.
	\end{defn}
	We comment that the case $\varphi(x)\neq\varphi(x')$ never occurs in the actions of $\Z$ equipped with the natural order. Thus, Definition~\ref{phi-asympt} generalizes the classical definition of an asymptotic pair.
	
	Now we formulate a useful lemma by G. Zhang, which slightly generalizes Lemma~4 in \cite{BHR}.
	\begin{lem}\cite[Lemma~3.7]{Zhang}\label{lemma_zhang}
		Let $(X,T)$ be a topological dynamical system (with the action of $\Z$), let $\mu$ be a $T$-invariant, Borel probability measure on $X$ and let $\Xi$ be a $T$-invariant sub-sigma-algebra of the Borel sigma-algebra $\Sigma_X$. Then:
		\begin{enumerate}[(i)]
			\item the measure-theoretic dynamical system $(X,\Sigma_X,T,\mu)$ admits a sigma-algebra $\Theta\supset \Xi$ such that:
			\begin{itemize}
				\item $\bigvee_{n\in\Z}T^{-n}(\Theta)\overset{\mu}{=}\Sigma_X$,\footnote{The joining of two sigma-algebras $\Theta_1$, $\Theta_2$, denoted by $\Theta_1\vee\Theta_2$, is the smallest sub-sigma-algebra of $\Sigma_X$ containing both $\Theta_1$ and $\Theta_2$. Analogously, we define the joining of a countable collection of sigma-algebras.}
				\item the relative Pinsker sigma-algebra\footnote{The relative Pinsker sigma-algebra of $(X,T)$ with respect to the $T$-invariant sub-sigma-algebra $\Xi$ is defined as the largest sub-sigma-algebra of $\Sigma_X$ consisting of measurable sets $A\subset X$ such that $h(\mu, T, \{A,A^c\}|\Xi) = 0.$} $\Pi_T(X|\Xi)\overset{\mu}{=}\bigcap_{n=0}^{+\infty}T^{-n}(\Theta^+)$,\footnote{By $\Theta^+$ we denote the joining $\Theta^+=\bigvee_{k\ge 1} T^{-k}(\Theta)$.}
				\item  any pair of distinct points belonging to the same atom of $\Theta^+$ is asymptotic;
			\end{itemize} 
			\item if $h(\mu,T|\Xi)>0$, then the equality $\Theta^{+}\overset{\mu}{=}\Sigma_X$ does not hold.
		\end{enumerate} 
	\end{lem}
	\begin{rem}\label{rem-tb}
		Lemma~\ref{lemma_zhang} is stated for any measure-theoretic $\Z$-action $(X,\Sigma_X,\mu,T)$, where $X$ is a compact metric space and $T$ is a homeomorphism on $X$. However, in the proof of the lemma, continuity of $T$ is not used, while compactness of $X$ is required only to guarantee the existence of a refining sequence of finite partitions $(\Q_n)_{n\in\N}$ such that the diameters of these partitions tend to $0$ as $n\to+\infty$. In fact, existence of such a sequence of finite partitions is guaranteed by the weaker condition that $X$ is totally bounded.
	\end{rem}
	Now we formulate and prove a series of theorems concerning the existence of $\varphi$-asymptotic pairs in topological systems of positive entropy.
	\begin{thm}\label{main_thm}
		Let $(X,G)$ be a topological $G$-action. Let $\mu$ be a $G$-invariant Borel probability measure on $X$ such that there exists a factor map $\varphi$ from $(X,\mu,G)$ to a multiorder $(\tilde{\O},\nu,G)$ (in other words, $(X,\mu,G,\varphi)$ is a multiordered system). Assume also that $h(\mu,G|\Sigma_{\tilde{\O}})>0$. Then, there exist $\varphi$-asymptotic pairs in $X$. Moreover, the set of points of $X$ which belong to $\varphi$-asymptotic pairs, has positive measure $\mu$.
	\end{thm}
	\begin{rem}\label{main_rem}
		We note that Theorem~\ref{main_thm} remains true if the space $X$ is metric totally bounded but not necessarily compact.
	\end{rem}
	\begin{proof}[Proof of Theorem~\ref{main_thm}]
		Let $\Sigma_X$ denote the Borel sigma-algebra on $X$. By Theorem~\ref{motooe}, $(X,\Sigma_X,\mu,G)$ is orbit equivalent to the $\Z$-action $(X,\Sigma_X,\mu,S)$ where $S$ is the successor map on $X$. By Theorem~\ref{eq_entropies} we have $h(\mu,S|\Sigma_{\tilde{\O}})=h(\mu,G|\Sigma_{\tilde{\O}})>0$. Thus, by Lemma~\ref{lemma_zhang},\footnote{In general, the map $S$ need not be continous but, as mentioned in Remark~\ref{rem-tb}, continuity is not necessary for the validity of Lemma~\ref{lemma_zhang}.} the system  $(X,\Sigma_X,\mu,S)$ admits a sigma-algebra ${\Theta\supset \Sigma_{\tilde \O}}$ such that any pair of points belonging to the same atom of the sigma-algebra $\Theta^+$ is asymptotic (w.r.t.\ the successor map~$S$) and $\Theta^{+}$ is a proper sub-sigma-algebra of $\Sigma_X$. Since $\Theta^{+}\supset \Sigma_{\tilde{\O}}$, the atoms of $\Theta^{+}$ are subsets of fibers of $\varphi$ (which are atoms of $\Sigma_{\tilde{\O}}$). Therefore, for every two points $x,x'$ belonging to the same atom of $\Theta^+$ we have $\varphi(x)=\varphi(x')$.
		
		Consider the set $A\subset X\times X$ consisting of all pairs $(x,x')$ asymptotic w.r.t.\ the map $S$. Observe that
		\begin{equation*}
			A=\bigcap_{l\ge 1}\bigcup_{n\ge 1}\bigcap_{k\ge n}(S\times S)^{-k}(\{(x,x')\in X\times X: d_X(x,x')<\tfrac{1}{l}\}).
		\end{equation*}
		Since the map $S$ is Borel-measurable, the set $A$ is Borel. Consider a pair $(x,x')\in A$, i.e.,\ such that
		\begin{equation}\label{conver10}
			\lim\limits_{k\to+\infty}d_{X}(S^k(x),S^k(x'))=0.
		\end{equation}
		By Theorem~\ref{motooe}, $S^k(x)=k^{\varphi(x)}(x)$ and $S^k(x')=k^{\varphi(x')}(x')$. Thus, the convergence \eqref{conver10} takes on the form
		\begin{equation}
			\label{conver11}\lim\limits_{k\to+\infty}d_X\bigl(k^{\varphi(x)}(x),k^{\varphi(x')}(x')\bigr)=0.
		\end{equation}
		Henceforth, the pair $(x,x')$ is $\varphi$-asymptotic if and only if it is asymptotic w.r.t.\ the map $S$ and $\varphi(x)=\varphi(x')$. Therefore, the set of all $\varphi$-asymptotic pairs is equal to
		\begin{equation*}
			\tilde{A}=A\cap\{(x,x')\in X\times X: \varphi(x)=\varphi(x')\}.
		\end{equation*}
		Observe that the set $\{(x,x')\in X\times X: \varphi(x)=\varphi(x')\}$ is the inverse image of $\triangle_{\tilde{\O}}=\{(\prec,\prec):\, \prec\,\in\tilde{\O}\}$ (which is obviously closed) via the map $\varphi\times\varphi$, hence it is Borel as well. Consequently $\tilde{A}$ is also Borel. The set $A_{\varphi}$ of all points $x\in X$ which belong to $\varphi$-asymptotic pairs is the projection of $\tilde{A}$ on the coordinate $X$ and, by the Measurable Projection Theorem, it is measurable with respect to every Borel probability measure on $X$.
		
		Assume that $\mu(A_{\varphi})=0$. Then, for $\mu$-almost every $x\in X$, $x$ does not belong to any {$\varphi$-asymptotic} pair. Therefore, for $\mu$-almost every $x\in X$, the atom of $\Theta^+$ containing $x$ must be a singleton. Consequently, since $X$ is a Lebesgue space, we have $\Theta^+\overset{\mu}{=}\Sigma_X$, which contradicts the fact that $\Theta^+$ is a proper sub-sigma-algebra of~ $\Sigma_X$. Therefore, $\mu(A_{\varphi})>0$.
	\end{proof}
	
	Obviously, in case $h(\nu,G)=0$, we have $h(\mu,G|\Sigma_{\tilde{\O}})=h(\mu,G)>0$. Hence, from Theorem~\ref{main_thm} we may deduce the following corollary.
	\begin{cor}
		Let $(X,G)$ be a topological $G$-action. Let $\mu$ be a $G$-invariant Borel probability measure on $X$, such that $(X,\mu,G,\varphi)$ is a positive-entropy multiordered system with a multiorder factor $(\tilde{\O},\nu,G)$ of (measure-theoretic) entropy zero. Then, the set of points of $X$ which belong to $\varphi$-asymptotic pairs, has positive measure $\mu$.
	\end{cor}
	Using Theorem~\ref{main_thm} for a specific multiordered system on the product space $X\times\tilde\O$, we obtain the main theorem of this section. In case of the one-element multiorder $(\{<\},\delta_{<},\Z)$ on $\Z$, our theorem reduces to the Blanchard-Host-Ruette Theorem. Our theorem implies the existence of asymptotic pairs in positive entropy $\Z$-actions not only along the natural order on $\Z$ but also along many non-standard ones (examples of such orders are given in \cite{DOWZ}).
	\begin{thm}\label{cor_thm}
		Let $(X,G)$ be a topological action of a countable amenable group $G$, of positive entropy. For any multiorder $(\tilde{\O},\nu,G)$ on $G$ and $\nu$-almost every $\prec\ \in\tilde{\O}$, there exists a $\prec$-asymptotic pair in $X$.
	\end{thm}
	\begin{proof}
		By the variational principle for $G$-actions \cite[Variational Principle~5.2.7]{MO}, there exists an ergodic measure $\mu$ on $X$, such that the measure-theoretic entropy $h(\mu,G)$ is positive. Recall that, from the measure-theoretic point of view, $\tilde{\O}$ can be identified with a totally bounded and Polish space $\mathbf{Bi}(\Z,G)$ of all anchored bijections from $\Z$ to $G$.
		
		By the standard ergodic decomposition argument, it suffices to prove the theorem for an ergodic measure $\nu$. Consider the product space $(X\times \tilde{\O},\Sigma_X\otimes\Sigma_{\tilde{\O}})$, where $\Sigma_X$ and $\Sigma_{\tilde{\O}}$ are the Borel sigma-algebras on $X$ and $\tilde{\O}$ respectively. Observe that $\Sigma_X\otimes \Sigma_{\tilde{\O}}$ is the Borel sigma-algebra on $X\times \tilde{\O}$ for the product topology on $X\times \tilde{\O}$ (this follows from the fact that both $X$ and $\tilde{\O}$ are separable metric spaces). The measure theoretic dynamical system $(X\times\tilde\O,\Sigma_X\otimes\Sigma_{\tilde{\O}},\mu\times\nu,G)$ is a multiordered system via the projection $\varphi$ on the second coordinate. Moreover, since $\Sigma_{\tilde{\O}}$ and $\Sigma_X$ are independent, we have
		\begin{equation*}
			h(\mu\times\nu,G|\Sigma_{\tilde\O})=h(\mu,G)>0.
		\end{equation*}
		
		Thus, we may apply the version of Theorem~\ref{main_thm} for totally bounded metric spaces (see Remark~\ref{main_rem}) to the multiordered system $(X\times \tilde{\O},\Sigma_X\otimes\Sigma_{\tilde{\O}},\mu\times\nu,G,\varphi)$ and obtain that the set $A_{\varphi}$ of points which belong to $\varphi$-asymptotic pairs in $X\times \tilde{\O}$ has positive measure $\mu\times\nu$. Note that a pair of distinct points $(x,\prec)$ and $(x',\prec')$ in $X\times\tilde{\O}$ is $\varphi$-asymptotic if $\prec\,=\,\prec'$, and $x\neq x'$ is a $\prec$-asymptotic pair in $(X,G)$.
		
		By the Measurable Projection Theorem, the projection $\varphi(A_{\varphi})$ (which equals to the set of the orders $\prec\,\in\tilde{\O}$ for which there exists a $\prec$-asymptotic pair in $(X,G)$) is measurable with respect to $\nu$ as an analytic set in a Polish space, and clearly has positive measure $\nu$. Furthermore, the set $\varphi(A_{\varphi})$ is $G$-invariant. Indeed, let $\prec\,\in\varphi(A_{\varphi})$ and $g\in G$ be arbitrary. Let $j\in\Z$ be such that $j^{\prec}=g$. Let $(x,x')$ be a $\prec$-asymptotic pair in $X$. Then, by \eqref{16}, for any $k\in\Z$, we have
		\begin{equation*}
			k^{g(\prec)}(g(x))=((k+j)^{\prec}\cdot g^{-1}\cdot g)(x)=(k+j)^{\prec}(x)
		\end{equation*}
		and analogously for $x'$. Hence,
		\begin{equation*}
			\lim\limits_{k\to+\infty}d_X(k^{g(\prec)}(g(x)),k^{g(\prec)}(g(x')))=\lim\limits_{k\to+\infty}d_X((k+j)^{\prec}(x),(k+j)^{\prec}(x'))=0.
		\end{equation*}
		Thus, $(g(x),g(x'))$ is a $g(\prec)$-asymptotic pair, hence $g(\prec)\in \varphi(A_{\varphi})$. Since $\nu(\varphi(A_{\varphi}))>0$, by ergodicity of the measure $\nu$, we have ${\nu(\varphi(A_{\varphi}))=1}$. This ends the proof.
	\end{proof}
	\section{Principal extensions and tiling-based multiorders}\label{s4}
	The aim of this section is to familiarize the reader with the notions of principal extensions, array systems, tiling systems and multiorders arising from ordered tiling systems.
	\subsection{Principal extensions and array systems}
	It is known that every topological dynamical system $(X,G)$ has a zero-dimensional extension $(\hat{X},G)$ which is \emph{principal}, i.e., for every Borel probability $G$-invariant measure $\nu$ on $\hat{X}$ we have $h(\nu,G|\Sigma_X)=0$, where $\Sigma_X$ is the Borel sigma-algebra on $X$ (see e.q. \cite[Theorem~3.2]{H}). If $\htop(X,G)=0$, then, by the variational principle for amenable group actions, we have $\htop(\hat{X},G)=0$. The zero-dimensional system $(\hat{X},G)$ can be represented as a \emph{binary array system}, that is, every point $x\in X$ can be viewed as an array of symbols $(x_{n,g})_{n\in\N,g\in G}$, where $x_{n,g}\in \{0,1\}$ for each $n\in\N$ and $g\in G$. In other words $(\hat{X},G)$ is represented as a subsystem of $\bigl(\prod_{n\in\N}\{0,1\}^G,G\bigr)$ on which $G$ acts by the shifts, i.e., $(h(x))_{n,g}=x_{n,gh}$ for all $n\in\N$ and $g,h\in G$.
	
	Fix some $n\in\N$ and let $F$ be a finite subset of $G$. By a \emph{block with $n$ floors and domain} $F$ we mean any finite array $B=(b_{i,g})_{i\in[1,n],g\in F}\in\prod_{i=1}^n\{0,1\}^{F}$. For two blocks $B$ and $B'$ we  will write $B\approx B'$ if they have the same number $n$ of floors and there exists $h\in G$ such that if $F$ denotes the domain of $B$, then $Fh$ is the domain of $B'$ and $b_{i,g}=b'_{i,gh}$ for each $i\in[1,n]$ and $g\in F$.
	We say that a block $B$ with domain $F$ and $n$ floors \emph{appears} in $\hat X$, if there exists $x\in\hat X$ such that for every $i\in[1,n]$ and $g\in F$ we have $x_{i,g}=b_{i,g}$. In this situation we will write $x|_{[1,n]\times F}=B$. The collection of all blocks with $n$ floors and domain $F$, which appear in $X$, will be denoted by $\B_n(F)$. Observe that since $\hat X$ is shift-\inv, any block $B$ appears in $\hat X$ together with all blocks $B'$ such that $B'\approx B$.
	\subsection{Tilings and tiling systems}
	Throughout this subsection we omit the proofs of the cited theorems and lemmas as they can be found in \cite{DHZ} and appendices of \cite{DOWZ}. We continue to denote by $G$ a countable amenable group.
	\begin{defn}
		A \emph{tiling} of $G$ is a partition $\T$ of $G$ into countably many finite sets $T$ (called \emph{tiles}), i.e., we have
		\begin{equation*}
			G=\bigsqcup_{T\in\T}T \text{\ \ (disjoint union)}.
		\end{equation*}
	\end{defn}
	\begin{defn}\label{proper_tiling}
		A tiling $\T$ is called \emph{proper} if there exists a finite collection $\S$ of finite subsets of $G$ (not necessarily different), called \emph{shapes}, such that each shape contains the unit $e$ of $G$ and for every $T\in\T$ there exists a shape $S\in\S$ satisfying $T=Sc$ for some $c\in T$.
	\end{defn}
	Although a proper tiling admits many choices of collections of shapes, we will always fix one collection $\S$ and one representation $\T\ni T\mapsto (S,c)$, where $T=Sc$. The objects $S$ and $c$ will be called the \emph{shape} and the \emph{center} of $T$, respectively. 
	In what follows, we will be dealing only with proper tilings, hence we will skip the adjective ``proper''. 
	
	Let $\T$ be a tiling with the collection of shapes $\S$. To every $S\in \S$ we assign bijectively the symbol $``S"$. Let $V=\{``S": S\in \S\}\cup\{``0"\}$. The tiling $\T$ can be represented as a symbolic element denoted by the same letter $\T\in V^G$, defined as follows:
	\begin{equation*}
		\T(g) = \begin{cases}
			``\!S", &\text{if } \text{ $g$ is the center of some tile of shape $S$},\\ 
			``0", & \text{otherwise ($g$ is not a center of any tile)}.
		\end{cases}
	\end{equation*}
	\begin{defn}
		Let $V$ be an alphabet defined as above for some finite collection $\S$ of finite sets. Let $\TT\subset V^G$ be a subshift such that every element $\T\in\TT$ represents a tiling with the collection of shapes contained in $\S$. We will call $\TT$ a \emph{dynamical tiling} and $\S$ the \emph{collection of shapes} of $\TT$.
	\end{defn}
	\begin{defn}
		Let $(\TT_k)_{k\in\N}$ be a sequence of dynamical tilings. By a \emph{system of tilings} we mean any topological joining\footnote{By
			a \emph{topological joining} of a sequence of dynamical systems $(X_k,G)$, $k\in\N$, denoted by $\bigvee_{k\in\N}X_k$, we mean any closed subset of the Cartesian product $\prod_{k\in\N}X_k$, which has full projections onto the coordinates $X_k$, $k\in\N$, and is invariant under the product action given by $g(x_1,x_2,\dots)=(g(x_1),g(x_2),\dots)$. We remark that the symbol $\bigvee_{k\in\N}X_k$ refers to many possible topological joinings.} $\TTT=\bigvee_{k\in\N}\TT_k$ (the elements $\bT\in\TTT$ have the form of sequences of tilings $\bT=(\T_k)_{k\in\N}$, where $\T_k\in\TT_k$ for every $k\in\N$). 
	\end{defn}
	By convention, we let $\mathsf T_0$ be the trivial tiling, whose only shape is the singleton $\{e\}$. This dynamical tiling has only one element, hence it represents the trivial one-point system. We will attach $\mathsf T_0$ as the zero entry to any system of tilings.
	\begin{defn}
		Let $\TTT=\bigvee_{k\ge0}\TT_k$ be a system of tilings and let $\S_k$ denote the collection of shapes of $\TT_k$. The system of tilings $\TTT$ is:
		\begin{enumerate}[a)]
			\item \emph{F{\o}lner}, if the union of the collections of shapes $\bigcup_{k\ge 0}\S_k$, arranged in a sequence, is a F{\o}lner sequence in $G$.
			\item \emph{congruent}, if for every $\bT=(\T_k)_{k\ge0}\in\TT$ and every $k\in\N$, each tile $T\in\T_k$ is a union of some tiles of $\T_{k-1}$;
			\item \emph{deterministic}, if it is congruent and for every $k\in\N$ and every $S\in\S_k$, there exist sets $C_{S'}(S)\subset S$ indexed by $S'\in\S_{k-1}$, which satisfy
			\begin{equation*}
				S=\bigsqcup_{S'\in\S_{k-1}}\bigsqcup_{c'\in C_{S'}(S)}S'c'
			\end{equation*}
			and for every $\bT=(\T_k)_{k\ge 0}\in \TTT$, whenever $T=Sc$ is a tile of $\T_{k}$, then
			\begin{equation*}
				T=Sc=\bigsqcup_{S'\in\S_{k-1}}\bigsqcup_{c'\in C_{S'}(S)}S'c'c
			\end{equation*}
			is the partition of $T$ by the tiles of $\T_{k-1}$.
		\end{enumerate}
	\end{defn}
	
	\begin{rem}
		If a system of tilings $\TTT$ is deterministic, then for each ${\bT=(\T_k)_{k\ge 0}\in\TTT}$, every tiling $\T_{k}$ \emph{determines} all the tilings $\T_{k'}$ with $k'<k$, and the assignment $\T_{k}\mapsto \T_{k'}$ is a topological factor map from $\TT_{k}$ onto $\TT_{k'}$. Therefore, the joining $\TTT$ is in fact the inverse limit
		\begin{equation*}
			\TTT=\overset\longleftarrow{\lim_{k\to\infty}}\mathsf{T}_k.
		\end{equation*}
	\end{rem}
	\begin{defn}\cite[Definition~A.8]{DOWZ}
		Any F{\o}lner and deterministic system of tilings $\TTT$ will be called simply a \emph{tiling system} (to distinguish it from less organized systems of tilings).
	\end{defn}
	\begin{thm}(\cite[Theorem~5.2]{DHZ}) Every countable amenable group $G$ admits a tiling system with topological entropy zero.
	\end{thm}
	\subsection{Tiling-based multiorders}
	Every tiling system can be equipped with a partial order, as described below. 
	For every $k\in\N$ and every $S\in\S_k$, let $C(S)=\bigsqcup_{S'\in\S_{k-1}}C_{S'}(S)$ be the set of all centers of the subtiles of $S$. In $C(S)$ we fix some ordering:
	\begin{equation*}
		C(S)=\{c_1^S,c_2^S,...,c_{l(S)}^S\},
	\end{equation*}
	where $l(S)=|C(S)|$. 
	Instead of saying ``order of centers of subtiles'' we will simply say ``order of subtiles''.
	\begin{defn}
		By an \emph{ordered tiling system} we mean a tiling system $\TTT$ with the ordering of subtiles established for each shape $S\in\bigcup_{k\in\N}\S_k.$
	\end{defn}
	Observe that the order of subtiles of shapes $S\in\S_k$ induces also an order of subtiles of any tile $T$ from any 
	$\T_k\in\TT_k$. Indeed, since every such tile $T$ has the form $T=Sc$, where $S\in\S_k$, then its subtiles are naturally ordered as follows:
	\begin{equation*}
		T=\bigsqcup_{i=1}^{l(S)}S'_ic_i^Sc,
	\end{equation*}
	where $S'_i$ is the unique $S'\in\S_{k-1}$ such that $c_i^S\in C_{S'}(S)$.
	By a straightforward induction, for each $\bT=(\T_k)_{k\ge 0}\in\TTT$, $k\in\N$ and $T\in\T_k$, the above ordering induces a linear order $\prec_T$ of the elements of $T$ in such a way that whenever $T'$ is a subtile of $T$ then $\prec_T$ agrees with $\prec_{T'}$ on $T'$. The order $\prec_T$ can be identified with a bijection between $[1,|T|]\subset\N$ and $T$, which we will write as follows:
	$$
	i\mapsto i^{\prec_T}, \ \ (i\in[1,|T|], \ i^{\prec_T}\in T).
	$$
	For $1\le i\le j\le |T|$ the symbol $[i,j]^{\prec_T}$ will stand for $\{i^{\prec_T}, 
	(i+1)^{\prec_T},\dots,j^{\prec_T}\}$. The set $[i,j]^{\prec_T}$ will be called an \emph{order interval within $T$}. The collection of all intervals within $T$ (of any length between $1$ and $|T|$) will be denoted by $\I^{T}$.

	\begin{defn}
		Let $\TTT$ be a tiling system and let $\bT=(\T_k)_{k\ge 0}\in\TTT$. We say that $\bT$ is \emph{in general position} if
		\begin{equation*}
			\bigcup_{k\in\N}T_k^e=G,
		\end{equation*}
		where $T_k^e$ denotes the central tile of $\T_k$, i.e., the unique tile $T\in\T_k$ such that $e\in T$.
	\end{defn}
	Observe that in any congruent system of tilings the central tiles always form an increasing sequence, so the above 
	union is always increasing.
	
	Now let $\TTT$ be an ordered tiling system and let $\bT=(\T_k)_{k\ge 0}\in\TTT$ be in general position. Then $\bT$ determines a linear order $\prec_{\bT}$ on $G$ as follows: Let $a,b\in G$ be distinct. There exists $k\ge 1$ such that both $a,b$ belong to a common tile  $T$ of $\T_k$ (indeed, eventually $a,b$ belong to a common central tile). Then $a\prec_{\bT}b$ if and only if $a\prec_T b$. Note that this definition is correct (i.e., does not depend on the index $k$ of the tiling $\T_k$ in which the ``common'' tile $T$ is found).
	
	One may verify, that for any $a,b\in G$ there are at most finitely many elements $g\in G$ such that $a\prec_{\bT} g \prec_{\bT} b$. Hence, for every $\T\in \TTT$ which is in general position, $(G,\prec_{\bT})$ is order-isomorphic to either $(\Z,<)$ or $(\N,<)$, or $(-\N,<)$ (the two latter cases occur only if the central tiles $T_{k-1}^e$ have, from some $k$ onward, the smallest or respectively the largest index among the subtiles of $T_k^e$).
	\begin{defn} Let $\TTT$ be an ordered tiling system. Any element $\bT\in\TTT$ for which $(G,\prec_{\bT})$ is order-isomorphic to $(\Z,<)$ (in other words, $\prec_{\bT}$ is an order of type $\Z$), is called \emph{straight}. Moreover, we denote:
		\begin{equation*}
			\TTT_{\mathsf{STR}}=\{\bT\in\TTT: \bT \text{ is straight}\}.
		\end{equation*}
	\end{defn}
	
	The following lemma shows that the set $\TTT_{\mathsf{STR}}$ is large in both topological and measure-theoretic sense.
	\begin{lem}(\cite[Lemma~B.4]{DOWZ})\label{4.14}
		Let $\TTT$ be an ordered tiling system. Then, the set $\TTT_{\mathsf{STR}}$ is invariant, residual and has measure $1$ for every Borel probability $G$-invariant measure on $\TTT$.
	\end{lem}
	\begin{defn}
		Let $\TTT$ be an ordered tiling system. The \emph{tiling-based multiorder} is the set $$\tilde{\O}_{\TTT}=\{\prec_{\bT}\,: \bT\in\TTT_{\mathsf{STR}}\}.$$
	\end{defn}
	\begin{rem} The set $\tilde{\O}_{\TTT}\subset \tilde{\O}$ is invariant under the action of $G$ given by~\eqref{ac} and the assignment $\TTT_{\mathsf{STR}}\ni\bT\mapsto\,\prec_{\bT}\,\in\tilde{\O}_{\TTT}$ intertwines the shift action of $G$ on $\TTT_{\mathsf{STR}}$ with the action of $G$ on $\tilde{\O}_{\TTT}$. Moreover, since, by Lemma~\ref{4.14}, $\TTT_{\mathsf{STR}}$ supports all invariant measures, for every invariant measure $\mu$ on $\TTT$ this assignment is a measure-theoretic factor map from $(\TTT,\mu,G)$ to the system $(\tilde{\O}_{\TTT},\nu,G)$ (where $\nu$ is the image of $\mu$), which is a multiorder in the sense of Definition~\ref{def_multiorder}.
	\end{rem}
	\begin{lem}\label{lem_invariant}
		Let $\TTT$ be an ordered tiling system. For any finite subset $K\subset G$ and any $\eps>0$ there exists $l_0$ such that for any $\bT=(\T_k)_{k\ge 0}\in\TTT$, any $k\in\N$ and any tile $T\in\T_k$ with $|T|\ge l_0$, every order interval $I\in \I^T$ of length at least $l_0$ is $(K,\varepsilon)$-invariant.
	\end{lem}
	\begin{proof}The claim follows from the two facts:
		\begin{enumerate}
			\item $\TTT_{\mathsf{STR}}$ is dense in $\TTT$ (see Lemma~\ref{4.14}), hence any order interval $I\in \I^T$ as in the formulation of the lemma appears as an order interval for some order $\prec_{\bT}$, where $\bT\in\TTT_{\mathsf{STR}}$.
			\item There exists $l_0$ such that for any $\bT\in\TTT_{\mathsf{STR}}$, any $\prec_{\bT}$-order interval of length at least $l_0$ is $(K,\varepsilon)$-invariant (see \cite[Theorem~B.3]{DOWZ}).
		\end{enumerate}
		\vspace{-12pt}
	\end{proof}
	\begin{cor}\label{tocor}
		Let $\TTT=\bigvee_{k\ge 0}\TT_k$ be an ordered tiling system and let $\S_k$ denote the collection of shapes of $\TT_k$.  Let $(X,G)$ be a binary array system of \tl\ entropy zero. Then, for any $\eps>0$ and any $n\in\N$, there exists $l_0$ such that for any $\bT=(\T_k)_{k\ge0}\in\TTT$, any $k\in\N$ and any $T\in\T_k$ with $|T|\ge l_0$, every order interval $I\in \I^T$ of length $l\ge l_0$ satisfies
		\begin{equation}\label{card_nolog}
			\#\B_n(I)<2^{\lfloor l\cdot \eps\rfloor},
		\end{equation}
		where $\B_n(I)$ denotes the collection of all blocks with $n$ floors and domain $I$ appearing in $X$.
	\end{cor}
	\begin{proof}
		Suppose this is not true. Then there exists an $\eps>0$ and $n\in\N$, an increasing \sq\ of natural numbers $l_m$, a \sq\ of tiles $T_m$ appearing in $\TTT$ (i.e., ${T_m\in\T_{k_m}}$ for some $\bT=(\T_k)_{k\ge0}\in\TTT$ and $k_m\ge1$) with $|T_m|\ge l_m$ and a sequence of order intervals $I_m\in\I^{T_m}$ of lengths $l_m$ such that
		$$
		\#\B(I_m)\ge2^{\lfloor l_m\cdot \eps\rfloor}.
		$$
		By Lemma \ref{lem_invariant}, the \sq\ $(I_m)_{m\in\N}$ is a F\o lner \sq\ in $G$. Then the above inequality implies that $h_{\mathsf{top}}(X)\ge\eps$,\footnote{It is well known that \tl\ entropy does not depend on the choice of a F\o lner \sq.} which is a contradiction.
	\end{proof}
	\begin{defn}\label{top_mult_sys}
		Let $\phi:X\to\TTT$ be a \tl\ factor map from a \tl\ $G$-action $(X,G)$ to an ordered tiling system $(\TTT,G)$. Let $\tilde{X}=\phi^{-1}(\TTT_{\mathsf{STR}})$ and let $\varphi: \tilde{X}\to\tilde{\O}_{\TTT}$ be given by $x\mapsto \prec_{\bT}$, where $\bT=\phi(x)$. The quadruple $(X,G,\phi,\varphi)$ will be called a \emph{topologically multiordered dynamical system}. A pair $(x,x')$ of distinct points in $X$ is called $\varphi$-asymptotic if $x,x'\in\tilde{X}$, $\varphi(x)=\varphi(x')$ (but not necessarily $\phi(x)=\phi(x')$) and $(x,x')$ is $\prec\,$-asymptotic, where $\prec\,=\varphi(x)$. 
	\end{defn}
	\begin{rem}
		Observe that if $(X,G,\phi,\varphi)$ is a topologically multiordered dynamical system and $\mu$ is an invariant measure on $X$, then, by Lemma~\ref{4.14}, $(X,\mu,G,\varphi)$ is a multiordered dynamical system in the sense of Definition~\ref{def_multiorder}. Moreover, every $\varphi$-asymptotic pair in the sense of Definition~\ref{top_mult_sys} is $\varphi$-asymptotic in the sense of Definition~\ref{phi-asympt}. In other words, the new definition encompasses a larger class of pairs.
	\end{rem}
	\subsection{Centered and odometric tiling systems}
	In this subsection we introduce two classes of tiling systems with additional convenient properties.
	\begin{defn}
		An ordered tiling system $\TTT=\bigvee_{k\ge 0}\TT_k$ is said to be \emph{centered} if, for any $\bT=(\T_k)_{k\ge0}\in\TTT$, any $k\in\N$ and any $T\in\T_k$, the order $\prec_T$ starts at the center of $T$ (see the text following Definition~\ref{proper_tiling})
	\end{defn}
	\begin{rem}\label{rem_centered}
		Any ordered tiling system $\TTT$ has a centered tiling system as a topological factor. Indeed, the factor map consists of appropriately moving the centers of all tiles. Note also that the factor map sends the straight elements into straight elements and preserves the resulting orders. 
	\end{rem}
	\begin{defn}\label{def_odometr}
		Let $(p_k)_{k\in\N}$ be a strictly increasing sequence of natural numbers such that $p_{k-1}|p_{k}$ for all $k\in\N$. An ordered tiling system $\TTT=\bigvee_{k\ge 0}\TT_k$ with the collections of shapes $\S_k$, $k\ge 0$, is called \emph{odometric} (with base $(p_k)_{k\in\N}$), if for every $\bT=(\T_k)_{k\ge 0}\in\TTT$, every $k\in\N$ and every tile $T\in\T_{k}$, the centers of all subtiles $T'\in\T_{k-1}$ of $T$, satisfy the condition
		\begin{equation}
			{j_{T'}\equiv j_T\!\!\mod{p_{k-1}}},
		\end{equation}
		where $j_{T'}$ and $j_T$ are the positions of the centers of $T'$ and $T$, respectively, counting along~$\prec_T$.
	\end{defn}
	\begin{rem}
		If $\TTT$ is odometric and $\bT=(\T_k)_{k\ge 0}\in\TTT_{\mathsf{STR}}$, then, for every $k\in\N$, the positions of centers of all tiles $T\in\T_k$ are congruent modulo $p_k$, counting along~$\prec_{\bT}$.
	\end{rem}
	\begin{lem}\label{lem_odometric}
		Let $\TTT=\bigvee_{k\ge 0}\TT_k$ be a centered tiling system on $G$ and let $(p_k)_{k\in\N}$ be as in Definition~\ref{def_odometr}. There exists a tiling system $\TTT'$ on $G$, odometric with base $(p_k)_{k\in\N}$, which is a principal extension of $\TTT$. Moreover, for any $\bT\in\TTT_{\mathsf{STR}}$, every $\bT'\in\TTT'$ in the fiber of $\bT$ is straight and we have $\prec_{\bT'}\,=\,\prec_{\bT}$.
	\end{lem}
	\begin{proof}
		By passing to a subsequence of $(\TT_k)_{k\ge 0}$, we may assume that $p_k\le \min\{|S|: S\in\S_k\}$. The construction of $\TTT'$ goes by induction. In the initial step, for every $k\in\N$ and every $\T_k\in\TT_k$, we move the center of each tile $T\in\T_k$ to one of the initial (w.r.t.\ $\prec_{T}$) $p_k$ positions of $T$. We do so in all $p_k$ possible ways, independently in all tiles. In this manner we create a congruent (but not deterministic) system of tilings $\TTT^{(0)}$, which clearly is a topological extension of $\TTT$. Each shape $S\in\S_k$ is replaced by $p_k$ new shapes, depending on the position of the center. 
		Topological entropy of the $k$-floor of this extension is increased by at most $\delta_k=\tfrac{\log{p_k}}{\min\{|S|: S\in\S_k\}}\le \frac{\log{p_k}}{p_k}$ and, since $p_k$ grows at least exponentially, we have $\sum_{k\in\N}\delta_k<\infty$. In step $n\ge 1$ we create $\TTT^{(n)}$ by keeping only those $(\T^{(0)}_{k})_{k\ge 0}\in\TTT^{(0)}$ for which, for any tile $T\in\T_{n}^{(0)}$, any $k\le n$ and any subtile $T'\in\T_{k}^{(0)}$ of $T$, we have ${j_{T'}\equiv j_T\!\!\mod{p_k}}$, where $j_{T'}$ and $j_T$ are the positions of the centers of $T'$ and $T$, respectively, counting along $\prec_T$. The conditional topological entropy of $\TTT^{(n)}$ given $\TTT$ is bounded by $\sum_{k\ge n}\delta_k$ (an amount which tends to $0$ as $n\to\infty$). The tiling system $\TTT'$ is defined as the intersection over $n$ of the congruent systems of tilings $\TTT^{(n)}$, $n\in\N$ ($\TTT'$ is now deterministic because the center of every tile determines the positions of the centers of all its subtiles). We skip the straightforward verification that $\TTT'$ enjoys the odometric property with base~$(p_k)_{k\in\N}$ and that for every $\bT\in\TTT_{\mathsf{STR}}$ all $\bT'$ in the fiber of $\bT$ are straight and satisfy $\prec_{\bT'}\,=\,\prec_{\bT}$. It is clear that the topological conditional entropy of $\TTT'$ w.r.t.\ $\TTT$ is zero, hence, it follows from the variational principle for conditional entropy (see e.g.\ \cite[Theorem~5.1]{Y}), that $\TTT'$ is a principal extension of $\TTT$.
	\end{proof}
	\section{Extensions with no asymptotic pairs of zero-entropy topological $G$-actions}\label{section_extension}
	This section is devoted to proving Theorem \ref{cor_extension}, which is a generalization of the inclusion $\mathsf{TEZ}\subset\mathsf{FNAP}$ of \cite[Theorem~4.1]{DL}, i.e., of the fact that any \tl\ system of entropy zero is a \tl\ factor of a system with no asymptotic pairs. Together with Theorem~\ref{cor_thm} this gives a full characterization of topological systems of entropy zero as factors of topological systems with no asymptotic pairs. Obviously, in the context of countable amenable groups, the notion of an asymptotic pair has to be replaced by either a $\prec$-asymptotic pair or a $\varphi$-asymptotic pair, which leads to slightly more complicated statements.
	
	By analogy to Section \ref{s3}, we will firstly formulate and prove the theorem for topologically multiordered systems, and then we extend it to arbitrary topological actions of countable amenable groups.
	
	Let $\phi:X\to\TTT$ be a \tl\ factor map from a \tl\ $G$-action $(X,G)$ to an ordered tiling system $(\TTT,G)$. Now suppose $\pi:Y\to X$ is a \tl\ factor map from another \tl\ system $(Y,G)$ to $(X,G)$ and denote $\tilde\phi=\phi\circ\pi$.
	The map $\tilde{\varphi}=\varphi\circ\pi$ is defined on the preimage $\tilde Y=\pi^{-1}(\tilde X)$ (recall that $\tilde X = \phi^{-1}(\TTT_{\mathsf{STR}})$). Then, the quadruple $(Y,G,\tilde\phi,\tilde\varphi)$ represents a topologically multiordered system. We say that a factor map $\pi$ collapses $\tilde\varphi$-asymptotic pairs if for every $\tilde\varphi$-asymptotic pair $(y_1,y_2)$ in $Y$ we have $\pi(y_1)=\pi(y_2)$.
	\smallskip
	
	The following very general and seemingly standard fact from \tl\ dynamics will be used in the proof of our next theorem. Since we were unable to find it in the literature, we provide it with an outline of a proof. The  zero-dimensionality assumption is unnecessary, but suffices for our needs and makes the proof significantly shorter.
	
	\begin{fact} \label{fct}
		Consider two $G$-actions $(Y,G)$ and $(X,G)$, where $X$ is zero-dimensional, and let $\pi:Y\to X$ be a continuous and equivariant map, which is not onto. Denote its image by $X'$ ($(X',G)$ is a subsystem of $(X,G)$ and a \tl\ factor of $(Y,G)$). Then there exists a system $(\bar Y,G)$ containing $(Y,G)$ and a \tl\ factor map $\bar\pi:\bar Y\to X$ such that $\bar\pi|_Y=\pi$. 
	\end{fact}
	
	\begin{proof} The multifunction $F=\pi^{-1}$ acting from $X'$ to $Y$ is upper semi-continuous, meaning that $F(x)$ is compact for any $x\in X'$ and, whenever $x_n\to x$ in $X'$, then 
		$$
		\bigcap_{n\ge1}\overline{\bigcup_{m\ge n} F(x_n)}\subset F(x).
		$$ 
		The graph of $F$, i.e., the set
		$$
		\{(x,y):x\in X', y\in F(x)\}\subset X'\times Y
		$$
		is homeomorphic to $Y$ via the projection on the second coordinate. So, we can replace $Y$ by this graph and then $\pi$ coincides with the projection on the first coordinate.  
		Since $\pi$ is a factor map, $F$ is \emph{equivariant}, i.e., for any $g\in G$ we have $F(gx)=g(F(x))$. Once we prolong $F$ to an upper semi-continuous and equivariant multifunction $F'$ from $X$ to $Y$, the proof will be ended by letting $\bar Y$ be the graph of $F'$ and letting $\bar\pi$ be the projection onto $X$. 
		
		Upper semi-continuity of $F$ is equivalent to the existence of a \sq\ of continuous multifunctions $(F_n)_{n\ge 1}$ from $X'$ to $Y$ such that, for each $x\in X'$, we have $F(x) = \bigcap_n F_n(x)$. Since $X'$ is zero-dimensional, each $F_n$ can be chosen piecewise constant (constant on atoms of some finite clopen partition of $X'$). Such a function can be easily prolonged to a piecewise constant multifunction $F'_n$ from $X$ to $Y$. Finally, the multifunction $F'$ defined on $X$ by $F'(x)=\bigcap_{g\in G}\bigcap_{n\ge1} g(F'_n(g^{-1}x))$ is upper semi-continuous on $X$, equivariant, and~on~$X'$ coincides with $F$.
	\end{proof}
	
	We are now in a position to formulate the series of theorems that lead to the main result of this paper, i.e., Theorem \ref{most_imp_thm}.
	
	\begin{thm}\label{thm_extension}
		Let $G$ be a countable amenable group. Let $(X,G,\phi,\varphi)$ be a topologically muliordered dynamical system. If $\htop(X,G)=0$, then there exists a zero-entropy topological extension $(Y,G)$ of $(X,G)$ via a map $\pi$, such that $\pi$ collapses all $\tilde\varphi$-asymptotic pairs, where $\tilde\varphi=\varphi\circ\pi$.
	\end{thm}
	\begin{proof}
		By passing to a zero-dimensional principal extension, we may assume that $X$ is a binary array system. By Remark~\ref{rem_centered} we may assume that $\TTT=\phi(X)$ is centered. Since $X$ has \tl\ entropy zero, so does the tiling system $\TTT$. Let $\TTT'$ be the odometric tiling system described in Lemma~\ref{lem_odometric}. Since $(\TTT',G)$ is a principal extension of $(\TTT,G)$, it has \tl\ entropy zero, as well.  The fiber product $(\check X,G)$ of $(X,G)$ and $(\TTT',G)$ over the common factor $(\TTT,G)$ also has \tl\ entropy zero. Since $(\check X,G)$ is a zero-dimensional extension of $(X,G)$, from now on we will focus on building an extension of $(\check X,G)$ (collapsing asymptotic pairs).  Note that $(\check X,G)$ factors onto both $(\TTT,G)$ and $(\TTT',G)$ (via the naturally understood maps $\check\phi$ and $\check\phi'$, respectively), but the compositions $x\mapsto\,\prec_{\bT}$ and $x\mapsto\,\prec_{\bT'}$, where $\bT=\check\phi(x)$ and $\bT'=\check\phi'(x)$, coincide. We denote this composition by $\check\varphi$. In what follows, we will always refer to the odometric tiling $\TTT'$ and the topologically multiordered system $(\check{X},G,\check\phi',\check\varphi')$, but we will simplify the notation and write $\TTT$ and $(X,G,\phi,\varphi)$ instead.
		
		By $\I$ we denote the collection of all order intervals associated to the tiles of the tiling system:
		\begin{equation*}
			\I=\bigcup_{\bT=(\T_k)_{k\ge 0}\in\TTT}\ \bigcup_{k\in\N}\ \bigcup_{T\in\T_k}\I^T.
		\end{equation*} 
		
		For every $n\in\N$, let $\varepsilon_n=2^{-n}$. Since $X$ has topological entropy zero, by Corollary \ref{tocor}, for every $n\in\N$, there exists $k_n$, such that for any $I\in\I$ of length $p_{k_n}$, the cardinality of the collection $\B_n(I)$ of blocks with $n$ floors and domain $I$ which appear in $X$, satisfies the inequality
		
		\begin{equation}\label{card_nolog2}
			\#\B_n(I)<2^{\tfrac{p_{k_n}}{2^n}}.
		\end{equation}
		
		Hence, for every $n\in\N$ there exist an injective function (a code) $\psi_n^I$ from $\B_n(I)$ to the collection of all binary blocks of length $\tfrac{p_{k_n}}{2^n}$ (for simplicity we assume that~$2^{n}|p_{k_n}$). To improve readability, from now on we will write $p_{n}$ and $\T_{n}$ instead of $p_{k_n}$ and $\T_{k_n}$, respectively.
		
		Fix an $x\in \tilde{X}=\phi^{-1}(\TTT_{\mathsf{STR}})$ and let $\bT=(\T_k)_{k\ge 0}=\phi(x)$. We will construct inductively a binary element $y_x$ (part of the preimage $\pi^{-1}(x)$), as follows. 
		
		In step $1$ we look only at the first floor of $x$. By the odometric property of $\TTT$ and since $\bT$ is straight, the centers of all tiles of $\T_1$ are congruent modulo $p_1$, counting along $\prec_{\bT}$. This allows us to partition the entire group $G$ into $\prec_{\bT}$\,-order intervals $\dots,I^1_{-2},I^1_{-1},I^1_0,I^1_1,I^1_2,\dots$ of length $p_1$ so that the centers of all tiles from $\T_1$ occupy the initial positions in the intervals to which they belong (to fix the enumeration of the intervals, we agree that $I^1_0\ni e$).
		
		For every $i\in\Z$ we look at the block $B=x|_{\{1\}\times I^1_i}$ and find its image $\hat{B}=\psi_1^{I^1_i}(B)$ (recall that $|\hat{B}|=\tfrac{p_{1}}{2}$). Let $R$ be the initial half of the interval $I^1_{i+1}$ (note that $|R|=|\hat{B}|$). We now define $y_x$ restricted to $R$ by rewriting the symbols of $\hat{B}$ preserving their order. After this step $y_x$ remains undefined on exactly half of each $\prec_{\bT}$\,-order interval of length $p_1$ (hence on a set of density $\tfrac{1}{2}$).
		
		Let $n\in\N$. Assume that after $n$ steps of the construction, $y_x$ is undefined on a fraction $(\tfrac{1}{2})^n$ of each $\prec_{\bT}$\,-order interval of length $p_n$. In step $n+1$ we look at $n+1$ floors of $x$. As in step 1, we partition the entire group $G$ into $\prec_{\bT}$\,-order intervals $\dots,I^{n+1}_{-2},I^{n+1}_{-1},I^{n+1}_0,I^{n+1}_1,I^{n+1}_2,\dots$ of length $p_{n+1}$ so that the centers of all tiles from $\T_{n+1}$ occupy the initial positions in the intervals to which they belong.
		
		For every $i\in\Z$ we look at the block $B=x|_{[1,n+1]\times I^{n+1}_i}$ and find its image $\hat{B}=\psi_{n+1}^{I^{n+1}_i}(B)$ (recall that $|\hat{B}|=\tfrac{p_{n+1}}{2^{n+1}}$). Let $R$ consist of the initial half of the elements of the interval $I^{n+1}_{i+1}$ on which $y_x$ remains undefined (note that $|R|=|\hat{B}|$). We define $y_x$ restricted to $R$ by rewriting the symbols of $\hat{B}$ preserving their order. After this step $y_x$ remains undefined exactly on a fraction of $(\tfrac 12)^{n+1}$ of each $\prec_{\bT}$\,-order interval of length $p_{n+1}$ (hence on a set of density $(\tfrac{1}{2})^{n+1}$).
		
		Eventually, after all steps are completed, $y_x$ may remain undefined on some set of density zero. By filling these positions in all possible ways, we obtain multiple versions of $y_x$ which form the set $Y_x\subset \{0,1\}^G$. Now let
		\begin{equation}\label{Y_wzor}
			Y=\overline{\bigcup_{x\in \tilde{X}}\bigcup_{y_x\in Y_x}(y_x,x)}\subset \{0,1\}^G\times X.
		\end{equation}
		It is obvious that $Y$ is invariant under the product action of $G$ and $Y$ projects onto the \inv\ subset $\overline{\tilde X}\subset X$ (the closure of $\tilde X$), which supports all \im s of $(X,G)$. Let us now temporarily assume that $\overline{\tilde X}=X$ (we will address the general case at the end of the proof). In this case, $(X,G)$ is a factor of $(Y,G)$ via the mapping $\pi: (y,x)\mapsto x$. Since $X$ is a binary symbolic-array system, $Y$ can be viewed as $X$ with one additional floor (floor number 0). It is not hard to see that if $x\in\tilde{X}$, then $\pi^{-1}(x)=Y_x$ (i.e.,\ no new elements of the preimage are created by taking the closure in \eqref{Y_wzor}). Since $\tilde{X}$ has full measure for every invariant measure on $X$ and, for every $x\in\tilde X$, every $y_x\in Y_x$ is determined by $x$ except on a set of density zero, the conditional (measure-theoretic) entropy of any \im\ $\nu$ on $Y$ with respect to $X$ is equal to zero. Using the variational principle twice (to $(X,G)$ and to $(Y,G)$), we conclude that $(Y,G)$ has \tl\ entropy zero, as claimed.
		
		It remains to show that $\pi$ collapses $\tilde\varphi$-asymptotic pairs.
		Ad absurdum, assume that $(y_x,x)$ and $(y_{x'},x')$ form a $\tilde{\varphi}$-asymptotic pair, where $x$ and $x'$ are distinct elements of $X$. Then, by definition, $x$ and $x'$ belong to $\tilde{X}$ and 
		$$
		\varphi(x)=\tilde\varphi((y_x,x))=\tilde\varphi((y_{x'},x'))=\varphi(x')=\,\prec
		$$
		is an order of type $\Z$. Let $\bT=(\T_k)_{k\ge 0}=\phi(x)$ and $\bT'=(\T'_k)_{k\ge 0}=\phi(x')$. Since $\bT$ and $\bT'$ generate the same order, there are two possibilities:
		\begin{enumerate}[a)]
			\item For each $k\in \N$, the positions of the centers of all tiles of $\T_k$ are congruent modulo $p_k$ to those of $\T'_k$ (counting along $\prec$);
			\item For some $k\in\N$, the positions of the centers of all tiles of $\T_k$ are congruent modulo $p_k$ to some $j$, while those of $\T'_k$ are congruent modulo $p_k$ to some $j'\not\equiv j\!\!\mod p_k$.
		\end{enumerate}
		It is immediately seen that in case b) the symbolic representations of $\T_k$ and $\T'_k$ are distal\footnote{Elements $x,x'$ of a topological $G$-action $(X,G)$ are called distal if there exists $\delta>0$ such that $d_X(g(x),g(x'))>\delta$ for all $g\in G$.} (they have no common centers of tiles), thus so are $\bT$ and $\bT'$, and as a consequence also $x$ and $x'$, and finally $(y_x,x)$ and $(y_{x'},x')$ are distal as well. This contradicts the assumption that $(y_x,x)$ and $(y_{x'},x')$ are $\tilde{\varphi}$-asymptotic. 
		
		In case a) we argue as follows. There exists a floor $n_0$ and a position $g_0\in G$ such that $x_{n,g_0}\neq x'_{n,g_0}$. By a), for each $n\ge n_0$, the partitions of $G$ into the intervals $I^n_i$ ($i\in\Z$) of length $p_n$, resulting from $\T_{n}$ and $\T'_{n}$, coincide. Let $I^n_{i_0}$ be the interval containing $g_0$. Then, $B=x|_{[1,n]\times I^n_{i_0}}\neq x'|_{[1,n]\times I^n_{i_0}}=B'$ and also $\psi^{I^n_{i_0}}_{n}(B)\neq\psi^{I^n_{i_0}}_{n}(B')$. As a consequence $y_x|_{I^n_{i_0+1}}\neq y_{x'}|_{I^n_{i_0+1}}$. Hence, there exists $g_n\in I^n_{i_0+1}$ such that $y_x(g_n)\neq y_{x'}(g_n)$. By the construction, the elements $g_n$ for different indices $n$ are distinct and all of them follow $g_0$ in the order $\prec$. So, $y_x$ and $y_{x'}$ are not $\prec$\,-asymptotic and consequently $(x,y_x)$ and $(x',y_{x'})$ are not $\tilde{\varphi}$-asymptotic, which ends the proof in case $\overline{\tilde X}=X$.
		
		If $\overline{\tilde X}\subsetneq X$ the above argument still defines an extension $\pi:Y\to\overline{\tilde X}$. By Fact \ref{fct},
		there exists a system $(\bar Y,G)$ with $\bar Y\supset Y$ and a \tl\ factor map $\bar\pi:\bar Y\to X$ such that $\bar\pi|_Y=\pi$. This extension of $(X,G)$ has \tl\ entropy zero and collapses $\tilde\varphi$-asymptotic pairs (where now $\tilde\varphi$ stands for $\varphi\circ\bar\pi$), because both of these properties depend only on the fibers over~$\tilde X$ (which are the same as for $\pi$).
	\end{proof}
	
	\begin{rem}
		Let us compare the proof of \cite[Lemma 4.2]{DL} with that of Theorem~\ref{thm_extension}. In the former, 
		we take an odometer to base $(p_k)_{k\ge1}$ and view each of its elements as a congruent \sq\ $(\P_k)_{k\ge1}$ of partitions of $\Z$ into intervals of length $p_k$, correspondingly. Note that so understood odometer is in fact a tiling system $\TTT$ of~$\Z$. By default, the subtiles of each tile are ordered by the natural order $<$. If we agree that the center of each tile is at its left end, then this tiling system has the odometric property. From here on the construction of the extension is a particular case of that in the proof of Theorem~\ref{thm_extension}. The only difference is that $<$ is an \inv\ order on $\Z$, which allows one to (implicitly) define $\prec_{\bT}\,=\,<$ for all $\bT\in\TTT$, not just for $\bT\in \TTT_{\mathsf{STR}}$ (which is a proper subset of $\TTT$, just as in the general case).  On the other hand, in the proof of Theorem~\ref{thm_extension}, the orders $\prec_{\bT}$ are total orders of type $\Z$ only for $\bT\in\TTT_{\mathsf{STR}}$, for which reason this proof does not fully generalize that in \cite{DL}.
		To fix this problem we need to slightly modify the proof of Theorem \ref{thm_extension}. Namely, whenever $\bT=(\T_k)_{k\ge1}$ is not of type $\Z$ but $G$ splits into two parts $G_1, G_2$ on which $\prec_{\bT}$ is of type $-\N$ and $\N$, respectively (i.e., $G_1$ has a maximal but not minimal element while $G_2$ has a minimal but not maximal element) and for any $k\ge1$ the centers of all tiles of $\T_k$ contained in $G_1$ are congruent modulo $p_k$ to those in $G_2$, we declare, for each $k\ge1$ the first tile of $\T_k$ covering $G_2$ to be the successor of the last tile covering $G_1$ (this leads to $\prec_{\bT}$ to be total and of type $\Z$) and treat $\bT$ as straight. From here on we continue the proof without further changes. After this modification, Theorem~\ref{thm_extension} fully covers \cite[Lemma 4.2]{DL} and also the following theorems from this paper cover analogous facts from \cite{DL}.  
	\end{rem}
	
	\begin{thm}\label{thm_extension2}
		Let $G$ be a countable amenable group. Let $(X,G,\phi,\varphi)$ be a topologically muliordered dynamical system. If $\htop(X,G)=0$, then there exists a topological extension $(Y,G)$ of $(X,G)$, via a factor map $\pi$, such that $(Y,G)$ has no $\tilde{\varphi}$-asymptotic pairs, where $\tilde{\varphi}=\varphi\circ\pi$.
	\end{thm}
	\begin{proof}
		By Theorem~\ref{thm_extension}, there exists a zero-entropy topological extension $(Y_1,G)$ of $(X,G)$, via a factor map $\pi_1$ which collapses $\tilde{\varphi}_1$-asymptotic pairs, where $\tilde{\varphi}_1=\varphi\circ\pi_1$. Then, applying Theorem~\ref{thm_extension} to the system $(Y_1,G)$ we obtain a zero-entropy extension $(Y_2,G)$ of $(Y_1,G)$ via a map $\pi_2$ which collapses $\tilde{\varphi}_2$-asymptotic pairs, where $\tilde{\varphi}_2=\tilde{\varphi}_1\circ\pi_2=\varphi\circ\pi_1\circ\pi_2$. We proceed in this manner obtaining a sequence of extensions $(Y_n,G)$, $n\in\N$ bound by the factor maps $\pi_n$ such that for every $n\ge 2$ the map $\pi_{n}$ collapses all $\tilde{\varphi}_{n}$-asymptotic pairs, where $\tilde{\varphi}_n=\tilde{\varphi}_{n-1}\circ\pi_n=\varphi\circ\pi_1\circ\pi_2\circ\dots\circ\pi_{n}$. The desired extension $Y$ is obtained as the inverse limit of the sequence of systems $(Y_n,G)$ (recall that the elements of $Y$ are sequences $y=(y_n)_{n\in\N}$, where, for every $n\in\N$, $y_n\in Y_n$ and $y_n=\pi_{n+1}(y_{n+1})$). Let $\pi$ be the factor map from $(Y,G)$ to $(X,G)$ and let $\tilde{\varphi}=\varphi\circ\pi$. Suppose that some distinct points $y,y'$ in $Y$ form a $\tilde{\varphi}$-asymptotic pair. Then, for every $n\ge 2$ either $y_n=y'_n$ or $(y_n,y'_n)$ is a $\tilde{\varphi}_n$-asymptotic pair in $Y_n$. Since $\pi_n$ collapses $\tilde{\varphi}_n$-asymptotic pairs, we have $y_{n-1}=y'_{n-1}$. As $n-1$ ranges over all natural numbers, we get $y=y'$, which shows that there are no $\tilde{\varphi}$-asymptotic pairs in~$(Y,G)$.
	\end{proof}
	
	Our next step is to generalize Theorem~\ref{thm_extension2} to general topological dynamical systems with actions of countable amenable groups (not necessarily topologically multiordered). 
	
	\begin{thm}\label{cor_extension}
		Let $G$ be a countable amenable group and let $\TTT$ be any ordered tiling system of entropy zero, and let $\tilde{\O}_{\TTT}$ be the associated tiling-based multiorder. Then, for every topological $G$-action $(X,G)$ of entropy zero, there exists an extension $(Y,G)$ of $(X,G)$, such that $(Y,G)$ has no $\prec$-asymptotic pairs for any $\prec\,\in\tilde{\O}_{\TTT}$.
	\end{thm}
	\begin{proof}
		The product system $(X\times\TTT,G)$ has entropy zero and is topologically multiordered via the maps $\phi$ and $\varphi$, where $\phi$ is the projection on the second cordinate and $\varphi((x,\bT))=\,\prec_{\bT}$, $x\in X$, $\bT\in\TTT_{\mathsf{STR}}$. Therefore, by Theorem~\ref{thm_extension2}, there exists a topological extension $(Y,G)$ of $(X\times\TTT,G)$ via a map $\pi$, such that there are no $\tilde{\varphi}$-asymptotic pairs in $(Y,G)$, where $\tilde{\varphi}=\varphi\circ\pi$. Thus, for every $\prec\,\in\tilde{\O}_{\TTT}$, there are no $\prec$-asymptotic pairs in $(Y,G)$. Clearly, $(Y,G)$ is an extension of $(X,G)$.
	\end{proof}

	\begin{rem} 
		Note that, by Theorem~\ref{cor_thm}, any extension $(Y,G)$ which satisfies the assertion of Theorem~\ref{thm_extension2} or of Theorem~\ref{cor_extension} (not only those constructed in the proofs) must have topological entropy zero.
	\end{rem}
	Eventually, by a straightforward combination of Theorems~\ref{cor_thm}~and~\ref{cor_extension}, we obtain the following characterization of zero-entropy topological actions of countable amenable groups, which generalizes the equality $\mathsf{TEZ}=\mathsf{FNAP}$ in \cite[Theorem~4,1]{DL}.
	\begin{thm}\label{most_imp_thm}
		Let $G$ be a countable amenable group and let $(X,G)$ be a topological $G$-action. The following conditions are equivalent:
		\begin{enumerate}[a)]
			\item $\htop (X,G)=0$;
			\item For any ordered tiling system $\TTT$ of entropy zero on $G$ there exists a topological extension $(Y,G)$ of $(X,G)$, such that for any $\prec\,\in\tilde{\O}_{\TTT}$, there are no $\prec$-asymptotic pairs in $(Y,G)$;
			\item There exists a multiorder $(\tilde{\O},\nu,G)$ and a topological extension $(Y,G)$ of $(X,G)$, such that the set of those $\prec\,\in\tilde{\O}$ for which there are no $\prec$-asymptotic pairs in $(Y,G)$ is of full measure $\nu$;
			\item There exists a multiorder $(\tilde{\O},\nu,G)$ and a topological extension $(Y,G)$ of $(X,G)$, such that the set of those $\prec\,\in\tilde{\O}$ for which there are no $\prec$-asymptotic pairs in $(Y,G)$ has positive measure $\nu$.
		\end{enumerate}
	\end{thm}
	\begin{proof}
		The impication from a) to b) is Theorem~\ref{cor_extension}. To deduce c) from b) it suffices to let $(\tilde{\O},\nu,G)$ be $(\tilde{\O}_{\TTT},\nu,G)$, where $\nu$ is any invariant measure on $\tilde{\O}_{\TTT}$. Condition d) follows trivially from c). The implication from d) to a) is obtained as follows: if $(X,G)$ has positive entropy, then any topological extension $(Y,G)$ of $(X,G)$ has positive entropy and Theorem~\ref{cor_thm}, applied to $(Y,G)$, leads to the negation of d).
	\end{proof}
	\section*{Acknowledgements}
	The authors would like to thank Guohua Zhang for inspiring remarks regarding the proofs presented in Section 3.
	
	Mateusz Wi\c{e}cek is supported by National Science Center, Poland (Grant No. 2021/41/N/ST1/02816).

\end{document}